\documentclass[a4paper]{article}
\usepackage{amsmath,amssymb,graphicx}

\newtheorem{prethm}{{\bf Theorem}}

\newenvironment{thm}{\begin{prethm}{\hspace{-0.5
               em}{\bf.}}}{\end{prethm}}

\newtheorem{prepro}[prethm]{Proposition}

\newtheorem{prelem}[prethm]{Lemma}
\newenvironment{lem}{\begin{prelem}{\hspace{-0.5
               em}{\bf.}}}{\end{prelem}}

\newtheorem{precor}[prethm]{Corollary}

\newtheorem{prerem}[prethm]{{\bf Remark}}
\newenvironment{rem}{\begin{prerem}\em{\hspace{-0.5
              em}{\bf.}}}{\end{prerem}}

\newtheorem{preexample}{{\bf Example}}

\newtheorem{preproof}{{\bf Proof.}}

\newenvironment{proof}[1]{\begin{preproof}{\rm
               #1}\hfill{$\Box$}}{\end{preproof}}

\newcommand{\noi}{\noindent}
\newcommand{\e}{\epsilon}
\newcommand{\m}{{\rm mult}}

\newcommand{\g}{{\cal G}}
\newcommand{\h}{{\cal H}}
\newcommand{\al}{\alpha}

\newcommand{\be}{\beta}
\newcommand{\si}{\simeq}
\renewcommand{\thefootnote}

\title{Graphs with few matching roots}
\author{\sc Ebrahim Ghorbani\\
{\small {Department of Mathematics, K.N. Toosi University of Technology,}}\\
{\small P.O. Box 16315-1618, Tehran, Iran}\\
 {\small { School of Mathematics, Institute
for Research in Fundamental Sciences (IPM),}} \\
 {\small { P.O. Box 19395-5746, Tehran, Iran}}\\
 {\tt\small e\_ghorbani@ipm.ir}}
\begin{document}
\maketitle

\begin{abstract}
We determine all graphs whose matching polynomials have at most five distinct zeros. As a consequence, we find  new families of graphs which are determined by their matching polynomials.

\vspace{3mm}
\end{abstract}

\section{Introduction}

All the graphs that we consider in this paper are finite,
simple and undirected. Let $G$ be  a graph. Throughout this paper
the {\it order} of $G$ is the number of vertices of $G$.
 A {\em $k$-matching} in $G$ is a set of $k$ pairwise nonincident edges and the number of $k$-matchings
in $G$ is denoted by  $m(G,k)$. If $G$ is of order $n$,  the {\em matching polynomial}
$\mu(G, x)$ is defined by
$$\mu(G, x)=\sum_{k\ge0}(-1)^km(G,k)x^{n-2k},$$
where $m(G, 0)$ is considered to be 1.
The roots of matching polynomial of any graph are all real numbers. (This was first proved independently in \cite{hl1} and \cite{ku}.)
 The matching polynomial is related to the characteristic polynomial
 of $G$, which is defined to be the characteristic polynomial of the adjacency matrix
of $G$. In particular these two coincide if and only if $G$ is a forest \cite{gg}. Also the
matching polynomial of any connected graph is a factor of the characteristic polynomial
of some tree (see \cite[Theorem~6.1.1]{gbook}). This is another way to see that the roots of matching polynomial are real numbers because the adjacency matrix of any graph is a symmetric matrix and so the roots of its characteristic polynomial are real numbers.
The roots of $\mu(G,x)$ are called the {\em matching roots} of $G$.
Two nonisomorphic graphs with the same matching polynomials are said to be {\em comatching}.
A graph $G$ is said to be {\em matching unique} if it has no comatching graph.
We also denote the multiset of the roots of the matching polynomial of $G$ by $R(G)$. We use exponent symbol to show the multiplicities of the elements of $R(G)$.

The determination of graphs with  few distinct roots of characteristic polynomials of matrices associated to graphs (i.e. graphs with few distinct eigenvalues) have been the subject of many researches.
Graphs with three adjacency eigenvalues have been studied by Bridges and
Mena \cite{bm}, Klin and Muzychuk \cite{km}, and van Dam \cite{dam1, dam2}.
Connected regular graphs with four distinct adjacency eigenvalues have
been studied by Doob \cite{do1, do2}, van Dam \cite{dam1}, and van Dam and Spence \cite{ds}.
Graphs with three Laplacian eigenvalues have been treated by van Dam and Haemers \cite{dh}.
 Ayoobi, Omidi and Tayfeh-Rezaie \cite{aot} investigated nonregular
graphs whose signless Laplacian matrix has three distinct eigenvalues.
For a complete survey on this subject see Chapter~14 of Brouwer and Haemers \cite{bh}.

So far, few families of graphs have been shown to be matching unique; these include unique
cages (regular graphs with minimum number of vertices and given degree and girth), 2-regular graphs,
$mK_{r,r}$, $mL$, where $L$ is a unique Moore graph with given degree and odd girth,
and the regular complete multipartite graphs \cite{n}.  It is also known that if a graph is matching unique,
then its complement is also matching unique (see \cite{bf}).

In this paper, we determine all graphs with at most five distinct matching roots. As a result, we find  new families of matching unique graphs. In particular, we show that for any positive integer $n\ne2$, the friendship graph $F_n$ (the graph consisting of  $n$ triangles intersecting in a single vertex) is matching unique.

\section{Graphs with few matching roots}

We denote the complete graph of order $n$ by $K_n$ and the complete bipartite graph with parts of sizes $r$ and $s$ by $K_{r,s}$. The graph $K_{1,s}$ is called a {\em star}. The multiplicity of $\theta$ as a root of
$\mu(G,x)$ is  denoted by $\m(\theta,G)$.

The roots of the matching polynomial of any graph, like those of characteristic polynomial, have the ``interlacing'' property (\cite{hl2}, see also \cite[Corollary 6.1.3]{gbook}):

\begin{lem}\label{inter} Let $G$ be a graph and $u$ be a vertex of that. Then the roots of $\mu(G-u,x)$ interlace those of $\mu(G,x)$, i.e. if
$\theta_1\ge\theta_2\ge\cdots\ge\theta_n$ and $\eta_1\ge\eta_2\ge\cdots\ge\eta_{n-1}$ are the matching roots of $G$ and $G-u$, respectively, then
$$\theta_1\ge\eta_1\ge\theta_2\ge\eta_2\ge\cdots\ge\theta_{n-1}\ge\eta_{n-1}\ge\theta_n.$$
Consequently,  $\m(\theta,G)$ differs from $\m(\theta,G-u)$ by at most one.
\end{lem}

Gallai's Lemma in matching theory asserts that if a graph $G$ is connected and for each vertex $u$ of $G$, the size of maximum matchings of $G$ and $G-u$ are the same, then $G-u$ has a perfect matching. In the language of matching polynomial, this is equivalent to say that if $\m(0,G-u)<\m(0,G)$ for every vertex $u$ of $G$, then $\m(0,G)=1$.
This result has been extended to any root of matching polynomials and is quoted as ``an analogue of Gallai's Lemma'':

\begin{thm} {\rm(Ku and Chen \cite{kuchen})} For a connected graph $G$, if  $\m(\theta,G)\ge2$, then there is a vertex $u$ of $G$ such that $\m(\theta,G-u)\ge\m(\theta,G)$.
\end{thm}

Godsil \cite{g} proved that if there exists two adjacent vertices $u,u'$ in $G$ such that $\m(\theta,G-u)\ge\m(\theta,G)>\m(\theta,G-u')$, then $\m(\theta,G-u)=\m(\theta,G)+1$.
He also observed that any graph $G$ has at least one vertex $v$ such that  $\m(\theta,G-v)=\m(\theta,G)-1$.
These results  together with the above theorem implies the following lemma.

\begin{lem}\label{mult} {\rm(\cite{kuchen})} For a connected graph $G$, if  $\m(\theta,G)\ge2$, then there is a vertex $u$ of $G$ such that $\m(\theta,G-u)=\m(\theta,G)+1$.
\end{lem}

\begin{rem} Any graph with an odd number of vertices has a zero matching root and if $\theta$ is a matching root of a graph, then so is $-\theta$.
\end{rem}

\begin{lem}\label{-1} Let $G$ be connected graph. If the roots of $\mu(G,x)$ are  $\ge-1$, then $G\simeq K_2$.
\end{lem}
\begin{proof}{Let $G$ be of order $n$. Since the roots of $\mu(K_{1,2},x)$ are $\{0,\pm\sqrt{2}\}$, by interlacing, $G$ has no $K_{1,2}$ as an induced subgraph. Thus $G$ must be the complete graph $K_n$.  If $n\ge3$, then, by interlacing, $\mu(K_n,x)$ has a root $\le-\sqrt{3}$ because $\mu(K_3,x)=x^3-3x$. This implies  $G\simeq K_2$.}
\end{proof}

We now define two families of graphs which will be used later.

\noi{\bf Definition.} We add a single vertex $u$ to the graph $rK_{1,k}\cup tK_1$ and join $u$ to the other vertices by $p$ edges so that the resulting graph is connected and $u$ is adjacent with exactly $q$ centers of the stars (for $K_{1,1}$ either of the vertices is considered as center).
 Clearly
  \begin{equation}\label{rtpk}
r+t\le p\le r(k+1)+t~~\hbox{and}~~~0\le q\le r.
\end{equation}
 We denote the set of graphs obtained in this way by $\g(r,k,t;p,q)$. See Figure~\ref{s(r,p,q)}.
   For any $G\in\g(r,3,t;p,q)$, we add $s$ copies of $K_3$ to $G$ and join them by $\ell$ edges to the vertex $u$ of $G$ to make a connected graph. Clearly $s\le \ell\le3s$.
  We denote the set of these graphs  by $\h(r,s,t;p,q,\ell)$.
\begin{figure}
\centering \includegraphics[width=7cm]{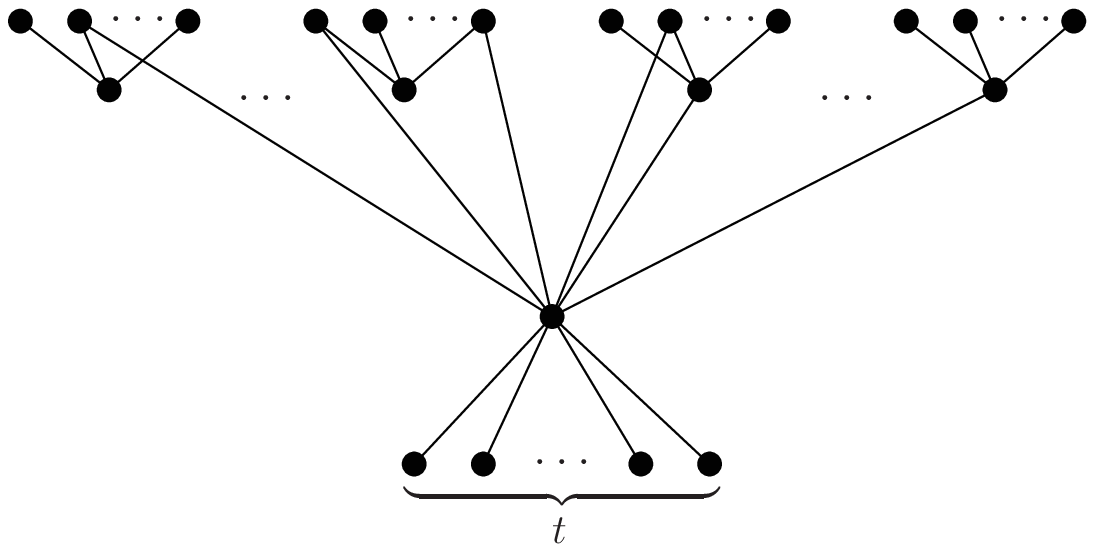}\\
  \caption{A typical graph in the family $\g(r,k,t;p,q)$: the central vertex has degree $p$ and it is joined to $q$ centers of the stars $K_{1,k}$.}\label{s(r,p,q)}
  \end{figure}

\begin{lem}\label{s(r,t)} For every $G\in\g(r,k,t;p,q)$, $$\mu(G,x)=x^{r(k-1)+t-1}(x^2-k)^{r-1}\left(x^4-(p+k)x^2+(p-q)(k-1)+t\right).$$
\end{lem}
\begin{proof}{Let $G\in\g(r,k,t;p,q)$. So $G$ has a vertex $u$ such that $G-u=rK_{1,k}\cup tK_1$. Since $\mu(K_{1,k},x)=x^{k-1}(x^2-k)$, by interlacing, $\{(\pm\sqrt k)^{r-1},\,0^{r(k-1)+t-1}\}\subseteq R(G)$. Let $\pm\al$ and $\pm\be$ be the remaining elements of $R(G)$. Since the squares of the roots of matching polynomial of a graph sum to its number of edges, we have $\al^2+\be^2+(r-1)k=rk+p$.
We note that the multiplicity of zero is equal to the number of vertices missed by a fixed maximum matching.
The maximum matching of  $G$ is of size $r$ if $t=0$ and $p=q$ and it is of size $r+1$ otherwise.
As the product of the squares of the nonzero roots of matching polynomial of a graph is equal to its number of maximum matchings,
$$\al^2\be^2k^{r-1}=m(G,\,r+1)=tk^r+(p-t-q)(k-1)k^{r-1}.$$
Note that if $m(G,r+1)=0$, the above equality is still valid; both sides are zero.
 Therefore,
$$\al^2+\be^2=p+k,~~\hbox{and}~~~\al^2\be^2=t+(p-q)(k-1).$$
The result now follows.
}\end{proof}
With the same arguments as in the proof of Lemma~\ref{s(r,t)}, we can prove the following.
\begin{lem}\label{h(r,t)} For every $G\in\h(r,s,t;p,q,\ell)$,
 $$\mu(G,x)=x^{2r+s+t-1}(x^2-3)^{r+s-1}\left(x^4-(p+\ell+3)x^2+3t+2(p-t-q)+\ell\right).$$
\end{lem}

We distinguish some special graphs in the families $\g$ and $\h$ which are important for our purpose.
The family $\g(r,1,0;s,q)$  consists of a single graph which we denote it by  $S(r,s)$. Note that in this case $q$ is determined by $r$ and $s$, namely $q=s-r$. Its matching polynomial is
\begin{equation}\label{muS}
    \mu(S(r,s),x)=x(x^2-s-1)(x^2-1)^{r-1}.
\end{equation}
The family $\g(r,k,0;r,r)$  consists of a single graph which is denoted  by  $T(r,k)$. Its matching polynomial is
\begin{equation}\label{muT}
\mu(T(r,k),x)=x^{r(k-1)+1}(x^2-r-k)(x^2-k)^{r-1}.
\end{equation}
 We also denote the unique graphs in $\g(1,k,t;\ell+t,0)$ and  $\g(1,k,t;\ell+t+1,1)$ by $K(k,t;\ell)$ and $K'(k,t;\ell)$, respectively.
 \begin{figure}
 $$\begin{array}{ccccc}
   \includegraphics[width=2.2cm]{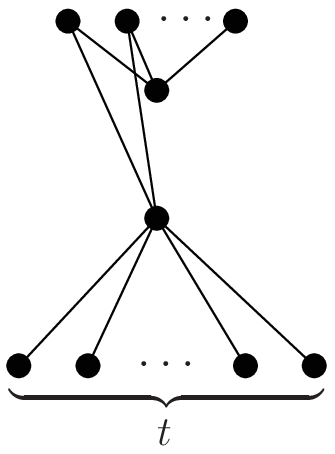} & \includegraphics[width=2.2cm]{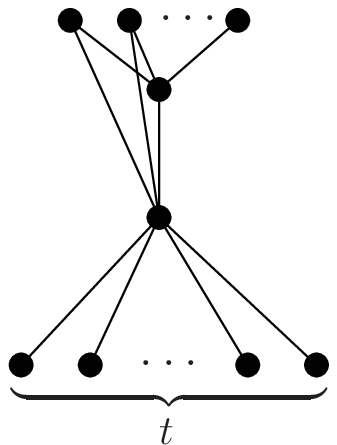} & \includegraphics[width=2.2cm]{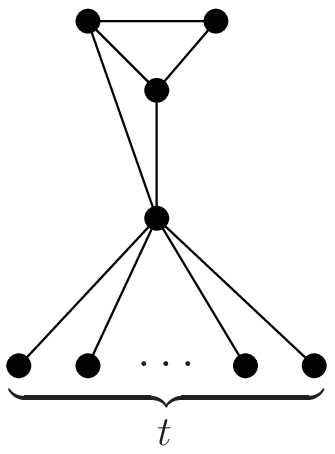} &
     \includegraphics[width=3.3cm]{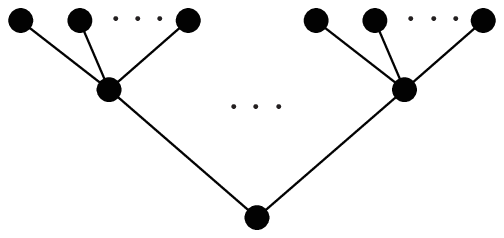}& \includegraphics[width=3.6cm]{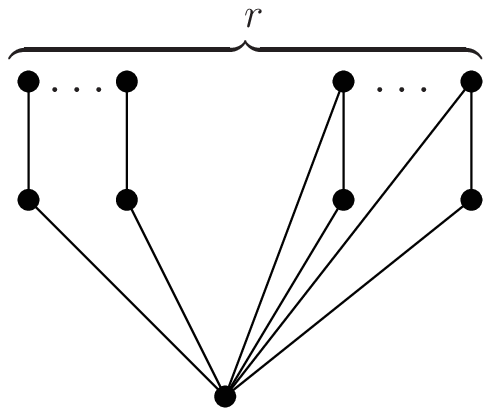}   \\
    K(k,t;\ell)& K'(k,t;\ell) & L(t,\ell)  &  T(r,k) & S(r,s)
 \end{array}$$
 \caption{The graphs $K(k,t;\ell)$, $K'(k,t;\ell)$,  $L(t,\ell)$, $T(r,k)$, and $S(r,s)$}\label{figK}
   \end{figure}
Their matching polynomials are
\begin{align} \mu(K(k,t;\ell),x)&=x^{k+t-2}\left(x^4-(k+t+\ell)x^2+(\ell+t)(k-1)+t\right),\label{muK}\\
 \mu(K'(k,t;\ell),x)&=x^{k+t-2}\left(x^4-(k+t+\ell+1)x^2+(\ell+t)(k-1)+t\right).\label{muK'}
\end{align}
Moreover, we show the unique graph in $\h(0,1,t;t,0,\ell)$ by $L(t,\ell)$ for $\ell=1,2,3$.
We have
\begin{equation}\label{muL}
\mu(L(t,\ell),x)=x^t\left(x^4-(t+\ell+3)x^2+3t+\ell\right).
\end{equation}
Typical graphs from the above families are demonstrated in Figure~\ref{figK}.

\begin{thm}\label{few} Let $G$ be a connected graph and $z(G)$ be the number of its distinct matching roots.
\begin{itemize}
  \item[\rm(i)] If $z(G)=2$, then $G\si K_2$.
    \item[\rm(ii)] If $z(G)=3$, then $G$ is either a star or $K_3$.
  \item[\rm(iii)] If $z(G)=4$, then $G$ is a non-star graph with $4$ vertices.
  \item[\rm(iv)] If $z(G)=5$, then $G$ is one of the graphs $K(k,t;\ell)$, $K'(k,t;\ell)$, $L(t,\ell)$, $T(r,k)$, $S(r,s)$, for some integers $k,r,s,t,\ell$, or a connected non-star graph with $5$ vertices.
\end{itemize}
\end{thm}

\begin{proof}{
(i) If $z(G)=2$, then $R(G)=\{(\pm\al)^r\}$, for some $\al\ne0$. Thus, from Lemma~\ref{mult} it follows that  $r=1$, and so $G\si K_2$.

(ii) If $z(G)=3$, then $R(G)=\{(\pm\al)^r,0^s\}$, for some $\al\ne0$. From Lemmas~\ref{mult} and \ref{inter} it is seen that it is impossible that $r\ge2$.
So $r=1$ and we have $m(G,2)=0$. Hence, $G$ is either the star $K_{1,s+1}$ or $K_3$.

(iii) If $z(G)=4$, then $R(G)=\{(\pm\al)^r,\,(\pm\be)^s\}$, for some nonzero $\al,\be$.
Again, Lemmas~\ref{mult} and \ref{inter} imply that $r=s=1$.
 So $G$ has four vertices possessing  a matching of size 2. Thus $G$ is a connected non-star graph on four vertices. 

(iv)  If $z(G)=5$, then $R(G)=\{(\pm\al)^r,\,(\pm\be)^s,\,0^t\}$, for some nonzero $\al,\be$.
Assume that $\be>\alpha>0$. If $s\ge2$, then, by Lemma~\ref{mult},
for some vertex $u$,  $R(G-u)=\{(\pm\al)^{r-1},\,(\pm\be)^{s+1},0^{t-1}\}$.
Hence $$m(G-u,s+r)=\be^{2s+2}\al^{2r-2}>\be^{2s}\al^{2r}=m(G,s+r),$$ which is a contradiction.
Therefore, $s=1$.

 If $t=r=1$, then $G$ is a graph on 5 vertices with $m(G,2)>0$. So $G$ can be any  connected non-star graph on 5 vertices.

If $t=1$ and $r\ge2$, then, by Lemma~\ref{mult}, there exists a vertex $u$ such that  $R(G-u)=\{(\pm\al)^{r+1}\}$.
From part (i) it follows that  $\al=1$ and  so $G-u\si(r+1)K_2$. Thus, for some $\ell$, $G\si S(r+1,\ell)$.
By (\ref{muS}), $\ell=\be^2-1$. Therefore, $G\si S(r+1,\be^2-1)$.

If $t\ge2$ and $r=1$, then there is a vertex $u$ such that $R(G-u)=\{\pm\eta,\,0^{t+1}\}$, for some $\eta\ne0$. Therefore, $G-u$ is either  $K_{1,k}\cup t'K_1$, for some $k$ and $t'=t-k+2$, or $K_3\cup tK_1$.
If $G-u\si K_{1,k}\cup t'K_1$, then either $G\si K(k,t';\ell)$ or $G\si K'(k,t';\ell)$, for some $\ell$.
If $G-u\si K_3\cup tK_1$, then $G\si L(t,\ell)$ for some $\ell\in\{1,2,3\}$.

If $t\ge2$ and $r\ge2$, then, for some vertex $u$, $R(G-u)=\{(\pm\al)^{r+1},\,0^{t-1}\}$. It turns out that $\al=\sqrt k$, for some integer $k$.
It is seen that one of the following cases may occur:
\begin{itemize}
\item[(a)]  $G-u=r'K_{1,k}\cup t'K_1$ with $r'=r+1$ and $t'=t-1-r'(k-1)$; or
\item[(b)] $G-u=r_1K_{1,3}\cup r_2K_3\cup t'K_1$ with  $r_2>0$, $r_1+r_2=r+1$ and $t'=t-2r_1-r_2-1$.
\end{itemize}
 If (a) occurs, then $G\in\g(r',k,t';p,q)$, for some $p,q$. Note that the polynomial $x^4-(p+k)x^2+(p-q)(k-1)+t'$ under the conditions (\ref{rtpk}) has no $\pm\sqrt k$ root and it has zero root if and only if $t'=0$ and either $k=1$ or $p=q$. This means that $G$ has five nonzero distinct matching roots if and only if $t'=0$ and either $k=1$  or $p=q$. If $k=1$, then $t'=t-1\ge1$ and thus $G$ has more than five distinct matching roots, a contradiction. If $p=q$ and $t'=0$, then $G\si T(r',k)$.
If (b) occurs, then
$G\in\h(r_1,r_2,t';p,q,\ell)$ for some integers $p,q,\ell$.
Note that since $r_2>0$, we have $\ell>0$. Then it is easily seen that the polynomial $x^4-(p+\ell+3)x^2+3t'+2(p-t'-q)+\ell$ has neither 0 nor $\pm\sqrt3$ as a root.
Therefore, $\mu(G,x)$ has more than five distinct roots which is a contradiction.
}\end{proof}

\section{Characterization by matching polynomial}

In this section we characterize the graphs $S(r,s)$, $L(t,\ell)$, $K'(1,t;1)$, $K(1,t;1)$, and $K(k,t;\ell)$ with $|\ell+t-k|\le1$  by their matching polynomials.

\begin{rem}\label{l+t=k} The graphs $K(k,t;\ell)$ and $K(\ell+t,k-\ell;\ell)$ are isomorphic. The same is true for the graphs $K'$.
\end{rem}

\begin{thm}\label{K(k,t,l)} The graphs $K(k,t;\ell)$ with $|\ell+t-k|\le1$ are matching unique except for  $$(k,t,\ell)\in\{(2,1,1),(2,1,2),(2,2,1),(3,0,2),(3,1,1),(3,1,2),(3,0,2),(3,2,2),(3,0,3),(4,1,2),(4,3,1)\}.$$
\end{thm}
\begin{proof}{ We first note that in a $K(k,t;\ell)$, $k\ge\ell$. By the assumption of the theorem, $\ell+t=k+\e$ for $\e\in\{-1,0,1\}$.
Let $G$ be a graph with $\mu(G,x)=\mu(K(k,t;\ell),x)$. If one removes the isolated vertices (if any) from $G$, the resulting subgraph $H$
must be one of the graphs described in Theorem~\ref{few}\,(iii),(iv), a union of two stars, or a union of an star and $K_3$.

If $H$ is a $K(m,s;q)$, then from (\ref{muK}) it is clear that
$$k+t\ge m+s,~~~2k+\e=m+s+q,~~~\hbox{and}~~(k+\e)(k-1)+t=(q+s)(m-1)+s.$$
Let $m=k+r$ for some integer $r$. Then $q+s=2k+\e-m=k+\e-r.$
Thus
\begin{align*}
    s&=(k+\e)(k-1)+t-(q+s)(m-1)\\&=t+(k+\e)(k-1)-(k+\e-r)(k+r-1)\\
    &=t+r^2-\e r-r.
\end{align*}
Therefore, $m+s=k+t+r^2-\e r$. Since $r^2-\e r\ge0$,  we have $m+s=k+t$ and $r^2-\e r=0$. So $r=0$ or $r=\e$. If $r=0$, then $m=k$, $s=t$, and $q=\ell$.
If $r=\e$, then $m=k+\e$, $s=t-\e$, and $q=\ell=k+\e-t$. But the graphs $K(k,t;k+\e-t)$ and $K(k+\e,t-\e;k+\e-t)$ are isomorphic by Remark~\ref{l+t=k}.

If $H$ is a $K'(m,s;q)$, then from (\ref{muK}) and (\ref{muK'}) we see that
$$k+t\ge m+s,~~~2k+\e=m+s+q+1,~~~\hbox{and}~~(k+\e)(k-1)+t=(q+s)(m-1)+s.$$
Let $m=k+r$ for some integer $r$. Then
\begin{equation}\label{q+s} q+s=k+\e-r-1.\end{equation}
Thus
\begin{align*}
    s&=(k+\e)(k-1)+t-(q+s)(m-1)\\&=t+(k+\e)(k-1)-(k+\e-r-1)(k+r-1)\\
    &=t+k+r^2-\e r-1.
\end{align*}
Therefore, $m+s=t+k+r^2-\e r+k+r-1$. On the other hand, $k+t\ge m+s$, $r^2-\e r\ge0$, and $k+r=m\ge1$.
So the equality must occur in all the above three inequalities.
It follows that $m=k=1$ and $r=0$. Hence, by (\ref{q+s}), $q+s=\e$. Since $q\ge1$, we find that $\e=q=1$ and $s=0$. Thus $t=0$ and $\ell=k+\e-t=2$, a contradiction.

 If $H$ is an $L(s,q)$, then, by (\ref{muL}),
 $$   k+t-2\ge s,~~~2k+\e=s+q+3~~~\hbox{and}~~(k+\e)(k-1)+t=3s+q.$$
From the last two equations  we have
$$2q+9=3(2k+\e)-(k+\e)(k-1)-t.$$
It follows that
\begin{align}
    2q+9+t&=4\e+(7-\e)k-k^2\label{q9t}\\
    &\le49/4+\e/2+\e^2/4.\nonumber
\end{align}
This holds only if $q\le2$. If $q=2$, then $t=0$ and $\e=1$ implying $\ell=k+1$, a contradiction. Therefore, $q=1$.
Now, it is easy to find all the values of $k,t,\e$ satisfying (\ref{q9t}). The comathing graphs obtained here are shown in Table~\ref{tab2}.

Now, if we consider $H$ to be a connected graph of order 4 or 5,
we come up with the right list of Table~\ref{tab2}.
\begin{table}
  \centering
  \begin{tabular}{cl}
\hline
   graph & comatching\\
   \hline
   $K(2,1;2)$ & $L(1,1)$\\
  $K(3,0;2)$ & $L(1,1)$\\
   $K(3,1;2)$ & $L(2,1)$\\
  $K(3,2;2)$ & $L(3,1)$\\
  $K(4,1;2)$ & $L(3,1)$\\
  $K(4,3;1)$ & $K_{1,5}\cup K_3$\\
  \hline
\end{tabular}
\hspace{1cm}
\begin{tabular}{cl}
\hline
   graph & comatching\\
   \hline
   $K(2,1;1)$ & $K_{1,1}\cup K_3$\\
   $K(2,1;2)$ & $5.16$\\
  $K(2,2;1)$ & $5.18\cup K_1$\\
  $K(3,0;2)$ & $5.16$\\
  $K(3,0;3)$ & $5.12$\\
  $K(3,1;1)$ & $5.18\cup K_1$\\
     \hline
\end{tabular}
  \caption{ The graphs $K(k,t;\ell)$ with $|\ell+t-k|\le1$ which are not matching unique.}\label{tab2}
\end{table}

If $H$ is a $K_{1,r}\cup K_{1,s}$, then we have $k+t\ge r+s$ and $\ell+t+k=r+s$ which implies that $\ell=0$, a contradiction.

If $H$ is a $K_{1,r}\cup K_3$, then
$$k+t-2\ge r,~~~\ell+t+k=r+3,~~~\hbox{and}~~(k+\e)(k-1)+t=3r.$$
The first two conditions imply that $\ell=1$, $t=k+\e-1$ and $r=2k+\e-3$.
Now, the third condition yield to $k^2+(\e-6)k-3\e+8=0$.
Hence $k=3-\e/2\pm\frac{1}{2}\sqrt{4+\e^2}$ and so $\e=0$ and $k=2,4$.
It follows that the graphs $K(2,1;1)$ and $K(4,3;1)$ are comatching with $K_{1,1}\cup K_3$ and $K_{1,5}\cup K_3$, respectively.

If $H$ is a $T(r,m)$, then $r=2$; since if $r\ge3$, then $T(r,m)$ has more than 4 non-zero roots and if $r=1$, then $T(r,m)$ is an star.
Thus, in view of (\ref{muT}),
$$k+t-2\ge2m-1,~~~2k+\e=2m+2,~~~\hbox{and}~~(k+\e)(k-1)+t=m(2+m).$$
The second condition implies that $\e=0$ and so $m=k-1$. Now the third condition gives $t=k-1$. This means that $K(k,k-1;1)$ and $T(2,k-1)$ have the same matching polynomial, but indeed they are isomorphic.

If $H$ is an $S(r,s)$, then with the same reason as above, $r\le2$. So $H$ is a graph with at most 5 vertices which is already considered.
}\end{proof}

We denote graph $K(1,t;1)$ and $K'(1,t;1)$ by $S(t)$ and $S'(t)$, respectively.

\begin{thm}\label{s(t)} For any integer $t\ge0$, the graph $ S(t)$ is matching unique unless $t\in\{2,3,4\}$; and $ S'(t)$ is matching unique unless $t\in\{2,3\}$.
\end{thm}
\begin{proof}{For $t=0,1$, we have  $ S(0)\si K_{1,2}$ and $ S(1)$ isomorphic to the path on 4 vertices which are matching unique.
Suppose that $t\ge2$ and
 $G$ be a graph with $$\mu(G,x)=\mu(S(t),x)=x^{t-1}(x^4-(t+2)x^2+t).$$
 It is easily seen that the polynomial $x^4-(t+2)x^2+t$ cannot be decomposed into $(x^2-r)(x^2-s)$ for some positive integers $r,s$.
Therefore, in view of Theorem~\ref{few}, $G$ consists of a connected component $H$ and possibly some isolated vertices such that $R(H)$ and $R( S(t))$ have the same nonzero elements. Whence $z(H)=4$ or $5$ and so $H$ is one of the graphs described in Theorem~\ref{few}\,(iii),(iv) excluding $S(r,s)$ and $T(r,k)$. Further, we have
\begin{equation}\label{m(H,2)}
    m(H,1)=m(H,2)+2,~~\hbox{and}~~t=m(H,2)\ge2.
\end{equation}

First, let $H$ be of order 4. From the Appendix table,  we see that, besides $4.6\si S(1)$, the only graph of order 4 satisfying (\ref{m(H,2)}) is
the graph 4.4, i.e. $K_{2,2}$, which corresponds to $t=2$. It follows that $ S(2)$ and $K_{2,2}\cup K_1$ are comatching.

Now let $H$ have 5 vertices. Among the connected graphs of order 5, from the Appendix table we see that the graphs
5.9 (corresponding to $t=4$), 5.15 (corresponding to $t=3$), and $5.20\si S(2)$ satisfy (\ref{m(H,2)}).
 Hence $S(3)$ and $S(4)$   are comatching with $5.15\cup K_1$ and $5.9\cup2K_1$, respectively.

If $H$ is an $L(s,q)$, then $t+2=s+q+3$ and $t=3s+q$ which implies that $2s-1=0$, a contradiction.

If $H$ is a $K(m,s;q)$, then $t+2=m+s+q$ and $t=(s+q)(m-1)+s$. Let $a=s+q$ and $b=m-1$. Then $s+ab=t=a+b-1$ and so
$$s+(a-1)(b-1)=0,~~~\hbox{with}~~a\ge1,~b\ge0.$$
If $a=1$, then $s=0$ and so $q=1$. Hence $m=t+1$ which gives the graph $K(t+1,0;1)$ which is isomorphic to $S(t)$.
So we assume that $a\ne1$. If $b=0$, then $m=1$ and $s+1-a=0$ which implies $q=1$ and $s=t$. This gives the graph $K(1,t;1)$ which is $S(t)$ itself. If $b=1$, then $m=2$ and $s=0$. So $1\le q=t\le2$.
If $q=t=1$, we have $K(2,0;1)$ isomorphic to $S(1)$ and if $q=t=2$, we obtain the graph $K(2,0;2)\cup K_1$ which is comatching with $S(2)$. If $b\ge2$, then $a=0$ implying  $q=0$  which is a contradiction.

If $H$ is a $K'(m,s;q)$, then $t+2=m+s+q+1$ and $t=(s+q)(m-1)+s$. Let $a=s+q$ and $b=m-1$. Then $s+ab=t=a+b$ and so
$$s+(a-1)(b-1)=1,~~~\hbox{with}~~a\ge1,~b\ge0.$$
If $a=1$, then $s=1$ and so $q=0$ which is impossible. So we assume that $a\ne1$. If $b=0$, then $m=1$ and $s+1-a=1$ implying $q=0$ which is again impossible.  If $b=1$, then $m=2$ and $s=1$. Hence either $q=1$, $t=3$ or $q=2$, $t=4$. It follows that $S(3)$ and $S(4)$ are comatching with $K'(2,1;1)\cup K_1$ and $K'(2,1;2)\cup2K_1$, respectively.
If $b=2$, then $m=3$ and $s=2-a=2-s-q$ implying that $2s=2-q$ which holds only if $q=2$ and $s=0$. Therefore, $t=4$ and we find that $S(4)$ is comatching with $K'(3,0;2)\cup2K_1$.
 If $b\ge3$, then $a=0$ which implies that $q=0$, a contradiction.

This completes the proof for $ S(t)$. The proof for $ S'(t)$ is similar.}
\end{proof}

\begin{rem} The graphs $L(t,1)$ and $L(t,3)$ are comatching with  $K(t+1,1;2)$ and $K'(t+2,0;3)$, respectively.
\end{rem}
\begin{thm} For any positive integer $t$, the graph $L(t,2)$  is  matching unique except for $t\in\{1,4,5,6\}$.
\end{thm}
\begin{proof}{ For $t=1$, the graph $L(t,2)$ is the graph 5.10 of the Appendix table and it is  comathing with 5.11.
  So we assume that $t\ge2$. 

Let $G$ be a graph with $$\mu(G,x)=\mu(L(t,2),x)=x^t\left(x^4-(t+5)x^2+3t+2\right).$$
If one removes the isolated vertices (if any) from $G$, the resulting subgraph $H$
must be one of the graphs described in Theorem~\ref{few}\,(iii),(iv), a union of two stars, or a union of an star and $K_3$.

 If $H$ is a $K_{1,r}\cup K_{1,s}$, then we have $r+s+2\le t+4$. On the other hand, counting the number of edges, $t+5=r+s$, a contradiction.
 If $H$ is a $K_{1,r}\cup K_3$, then we have $t+5=r+3$ and $3t+2=3r$ which is impossible.
 If $H$ is an $S(r,s)$, then $\pm1\in R(H)$ which is not the case for $t\ge2$.
 If $H$ is  a $T(r,k)$, then as in the proof of Theorem~\ref{K(k,t,l)}, we have $r=2$. Thus from (\ref{muT}) we see that $t\ge2k-1$ and $t+5=2k+2$, a contradiction.
 If $H$ is  an $L(t',\ell')$, then we have $t+5=t'+\ell'+3$ and $3t+2=3t'+\ell'$ which obviously implies that $t=t'$ and $\ell'=2$.

 If $H$ is  a $K(m,s;q)$, then
 $$t\ge m+s-2,~~~t+5=m+s+q~~~\hbox{and}~~3t+2=(s+q)(m-1)+s.$$
  From the first two conditions we see that $q\ge3$ and from the last two conditions, $1=(s+q-3)(4-m)-s$. It is straightforward to see that this equation has no solution under the condition $m\ge q\ge3$ and $s\ge0$.

 If $H$ is  a $K'(m,s;q)$, then
  $$t\ge m+s-2,~~~t+5=m+s+q+1~~~\hbox{and}~~3t+2=(s+q-1)(m-1)+s.$$
   From the first two conditions we see that $q\ge2$ and from the last two conditions,
 $$(s+q-3)(4-m)-s+2=0 ~~\hbox{subject to}~~m\ge q\ge2~~\hbox{and}~~s\ge0.$$
It is not hard to find all the solutions of the above equation.
It turns out that the graphs $L(4,2)$,  $L(5,2)$ and  $L(6,2)$ are comatching with $K'(4,2;2)$, $K'(4,2;3)\cup K_1$ and $K'(4,2;4)\cup2K_1$, respectively. Note that $L(6,2)$ is also comatching with $K'(5,0;5)\cup3K_1$.

Now, if we consider $H$ to be a connected graph of order 4 or 5,
we find that the graphs $L(1,2)$ and $L(2,2)$ are comatching with $5.11$ and $5.8\cup K_1$.}
\end{proof}

We remark that if $r\le r'\le s\le2r$, then the graphs $ S(r',s)$ and $ S(r,s)\cup(r'-r)K_2$ have the same matching polynomials. So in general, for a given $s$, the graphs $ S(r,s)$ are not matching unique unless for the smallest value of $r$ such that $ S(r,s)$ can be defined (these include the friendship graph $F_n\si S(n,2n)$). This is shown below.

\begin{thm}\label{s(s,s)} For any positive integer $s$,  the graph $ S(\lceil\frac{s}{2}\rceil,s)$  is  matching unique except for $s\in\{3,4,5\}$. In particular, for any positive integer $n\ne2$,  the friendship graph $F_n$ is matching unique.
\end{thm}
\begin{proof}{For $s=1,2$, the graphs $ S(\lceil\frac{s}{2}\rceil,s)$ are isomorphic to $K_{1,2}$ and $K_3$, respectively, which are obviously matching unique. So let $s\ge3$.
Let $G$ be a graph with $\mu(G,x)=\mu( S(\lceil\frac{s}{2}\rceil,s),x)$. Thus $\m(0,G)=1$ and $\pm\sqrt{s+1}\in R(G)$ for $s\ge3$.
It follows that $G$ has neither $K_3$ nor $K_{1,s+1}$ as a component.
Therefore from Theorem~\ref{few} it is seen that $G$ consists of a connected component $H$ and possibly some copies of $K_2$ and at most one isolated vertices such that $\{\pm1,\pm\sqrt{s+1}\}\subseteq R(H)$.
From the Appendix table we find that no graph $H'$ with $z(H')=4$ has $\pm1$ as roots of its matching polynomial. Hence $z(H)=5$, and
\begin{equation}\label{R(H)}
 R(H)=\{0,\pm1,\pm\sqrt{s+1}\},~~\hbox{for some $s\ge3$}.
\end{equation}
 Thus, $H$ is one of the graphs of Theorem~\ref{few}\,(iv).
 If $H$ is an $S(r',s')$, then it readily follows that $s'=s$ and $r'=\lceil\frac{s}{2}\rceil$.
  If $H$ is a graph of Theorem~\ref{few}\,(iv), not an $S(r,s)$, with more than 5 vertices, then $\m(0,H)\ge2$ which is impossible.
  It follows that $H$ is a non-star 5-vertex graph.
 From the Appendix table we see that only graphs of order 5 satisfying (\ref{R(H)}) are the graphs 5.5, 5.6, 5.10, 5.11, 5.16, and 5.17 of that table.
 The graph 5.16 is isomorphic to $ S(2,3)$, and 5.11 is isomorphic to $ S(2,4)$. It turns out  that the graph $ S(2,4)$ is comatching with 5.10, $ S(2,3)$ is comatching with 5.17, and $ S(3,5)$ is comatching with the union of a $K_2$ with either 5.5 or 5.6.
}\end{proof}

\noindent{\bf Acknowledgements.} The research of the author was in part supported by a grant from IPM (No. 90050117). The author is grateful to the referees whose helpful comments improved the presentation of the paper.

\begin{table}
\textbf{\large Appendix.}~Connected graphs up to five vertices and their matching polynomial

\vspace{.8cm}

\centering
{\small
  \begin{tabular}{ccc}
\hline
   label& graph & matching polynomial  \\
   \hline
     2.1&$\begin{array}{c} \includegraphics[width=1.3cm]{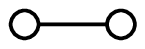}\end{array}$&$x^2-1 $ \\  \hline \vspace{-.3cm}
 3.1& $\begin{array}{c}\includegraphics[width=1.4cm]{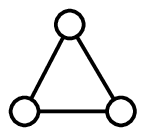}\end{array}$ &$x^3-3x$ \\ \vspace{-.2cm}
  3.2&$\begin{array}{c}\includegraphics[width=1.4cm]{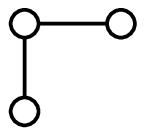}\end{array}$ &$x^3-2x $ \\ \hline \vspace{-.3cm}
4.1& $\begin{array}{c}\includegraphics[width=1.4cm]{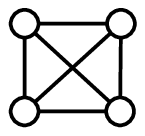} \end{array}$&$x^4-6x^2+3 $ \\ \vspace{-.3cm}
 4.2&$\begin{array}{c} \includegraphics[width=1.4cm]{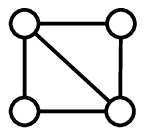} \end{array}$&$x^4-5x^2+2 $ \\ \vspace{-.3cm}
4.3&$\begin{array}{c}  \includegraphics[width=1.4cm]{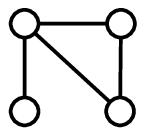}\end{array}$&$x^4-4x^2+1 $ \\ \vspace{-.3cm}
4.4& $\begin{array}{c} \includegraphics[width=1.4cm]{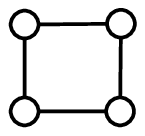} \end{array}$&$x^4-4x^2+2 $ \\ \vspace{-.3cm}
 4.5&$\begin{array}{c}\includegraphics[width=1.4cm]{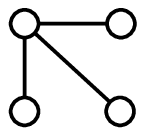} \end{array}$&$x^4-3x^2 $\\ \vspace{-.2cm}
  4.6&$\begin{array}{c}\includegraphics[width=1.4cm]{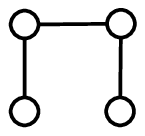}\end{array}$&$x^4-3x^2+1 $\\ \hline \vspace{-.4cm}
 5.1& $\begin{array}{c}\includegraphics[width=1.7cm]{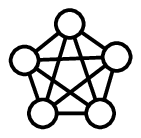}\end{array}$&$x^5-10x^3+15x$ \\ \vspace{-.55cm}
  5.2& $\begin{array}{c}  \includegraphics[width=1.7cm]{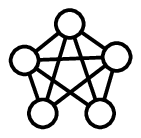} \end{array}$&$x^5-9x^3+12x $ \\ \vspace{-.55cm}
  5.3& $\begin{array}{c} \includegraphics[width=1.7cm]{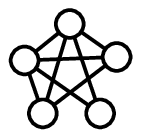} \end{array}$&$x^5-8x^3+9x $ \\ \vspace{-.55cm}
  5.4& $\begin{array}{c} \includegraphics[width=1.7cm]{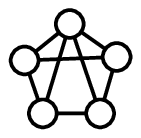}\end{array}$&$x^5-8x^3+10x$ \\ \vspace{-.55cm}
  5.5& $\begin{array}{c} \includegraphics[width=1.7cm]{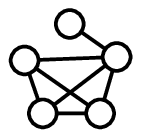}\end{array}$&$x^5-7x^3+6x$ \\ \vspace{-.2cm}
   5.6& $\begin{array}{c}\includegraphics[width=1.7cm]{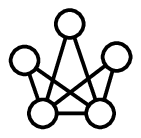}\end{array}$&$x^5-7x^3+6x$ \\ \hline
 \end{tabular}
 \hspace{.5cm}
  \begin{tabular}{ccc}
\hline
  label& graph & matching polynomial \\
   \hline
   \vspace{-.6cm}
    5.7&  $\begin{array}{c}\includegraphics[width=1.7cm]{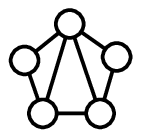}\end{array}$ &$x^5-7x^3+7x$ \\ \vspace{-.6cm}
   5.8& $\begin{array}{c}\includegraphics[width=1.7cm]{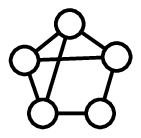}\end{array}$ &$x^5-7x^3+8x$ \\ \vspace{-.6cm}
 5.9& $\begin{array}{c} \includegraphics[width=1.7cm]{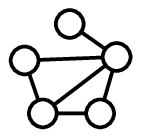}\end{array}$ &$x^5-6x^3+4x$ \\ \vspace{-.6cm}
  5.10&$ \begin{array}{c} \includegraphics[width=1.7cm]{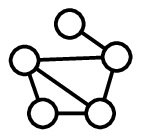} \end{array}$&$x^5-6x^3+5x$ \\ \vspace{-.6cm}
  5.11& $\begin{array}{c} \includegraphics[width=1.7cm]{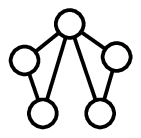}\end{array}$&$x^5-6x^3+5x$ \\ \vspace{-.6cm}
  5.12&$ \begin{array}{c} \includegraphics[width=1.7cm]{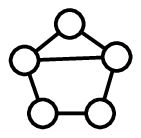}\end{array}$ &$x^5-6x^3+6x$ \\ \vspace{-.6cm}
 5.13& $\begin{array}{c} \includegraphics[width=1.7cm]{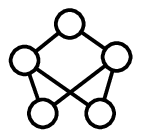}\end{array}$ &$x^5-6x^3+6x$\\ \vspace{-.6cm}
 5.14& $ \begin{array}{c} \includegraphics[width=1.7cm]{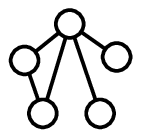}\end{array}$&$x^5-5x^3+2x$ \\ \vspace{-.6cm}
  5.15&$\begin{array}{c}  \includegraphics[width=1.7cm]{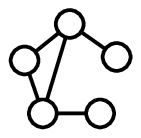}\end{array}$& $x^5-5x^3+3x$\\ \vspace{-.6cm}
  5.16& $\begin{array}{c} \includegraphics[width=1.7cm]{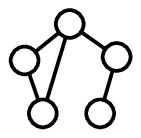}\end{array}$ &$x^5-5x^3+4x$\\ \vspace{-.6cm}
  5.17& $\begin{array}{c} \includegraphics[width=1.7cm]{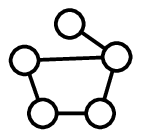}\end{array}$ &$x^5-5x^3+4x$\\ \vspace{-.6cm}
  5.18& $\begin{array}{c} \includegraphics[width=1.7cm]{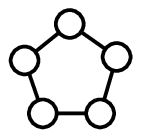}\end{array}$& $x^5-5x^3+5x$\\\vspace{-.6cm}
  5.19& $\begin{array}{c} \includegraphics[width=1.7cm]{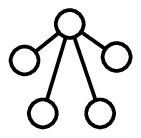}\end{array}$& $x^5-4x^3$\\ \vspace{-.6cm}
   5.20&$ \begin{array}{c} \includegraphics[width=1.7cm]{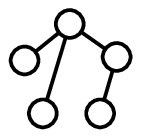} \end{array}$&$x^5-4x^3+2x$ \\ \vspace{-.3cm}
     5.21&$ \begin{array}{c}\includegraphics[width=1.7cm]{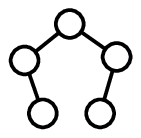} \end{array}$&$x^5-4x^3+3x$\\ \hline
 \end{tabular}
}\end{table}


\begin{thebibliography}{MM}
\bibitem{aot}F. Ayoobi, G.R. Omidi, and B. Tayfeh-Rezaie, A note on graphs whose signless
Laplacian has three distinct eigenvalues, {\em  Linear Multilinear Algebra} {\bf59} (2011), 701--706.
\bibitem{bf} R.A. Beezer and E.J. Farrell, The matching polynomial of a regular graph, {\em Discrete Math.} {\bf137} (1995), 7--18.

\bibitem{bm} W.G. Bridges and R.A. Mena, Multiplicative cones---a family of three eigenvalue graphs, {\em Aequationes Math.}  {\bf22} (1981), 208--214.

\bibitem{bh}A.E. Brouwer and W.H. Haemers, {\em Spectra of Graphs}, Springer, 2012.
\bibitem{dam1} E.R. van Dam, {\em Graphs with Few Eigenvalues. An Interplay between Combinatorics and Algebra}, Center Dissertation Series 20, Ph.D. Thesis, Tilburg University, 1996.

\bibitem{dam2} E.R. van Dam, Nonregular graphs with three eigenvalues, {\em  J. Combin. Theory Ser. B} {\bf73} (1998), 101--118.

\bibitem{dh} E.R. van Dam and W.H. Haemers, Graphs with constant $\mu$ and $\bar\mu$, {\em Discrete Math.} {\bf182} (1998), 293--307.

\bibitem{ds} E.R. van Dam and E. Spence, Small regular graphs with four eigenvalues, {\em Discrete Math.} {\bf189} (1998), 233--257.

\bibitem{do1} M. Doob, Graphs with a small number of distinct eigenvalues, {\em Ann. N.Y. Acad. Sci.} {\bf175} (1970), 104--110.

\bibitem{do2} M. Doob, On characterizing certain graphs with four eigenvalues by their spectrum, {\em Linear Algebra Appl.} {\bf3} (1970), 461--482.

\bibitem{gbook} C.D. Godsil, {\em Algebraic Combinatorics}, Chapman and Hall Mathematics Series, Chapman \& Hall, New York,  1993.

\bibitem{g} C.D. Godsil, Algebraic matching theory, {\em Electron. J. Combin.} {\bf2} (1995), \#R8.

\bibitem{gg} C.D. Godsil and I. Gutman, On the theory of the matching polynomial, {\em J. Graph Theory} {\bf5} (1981), 137--144.

\bibitem{hl1}O.J. Heilmann, E.H. Lieb, Monomers and dimers, {\em Phys. Rev. Letters} {\bf24} (1970), 1412--1414.

\bibitem{hl2} O.J. Heilmann, E.H. Lieb, Theory of monomer-dimer systems, {\em Comm. Math. Phys.} {\bf25}  (1972), 190--232.

\bibitem{km} M. Klin and M. Muzychuk, On graphs with three eigenvalues, {\em Discrete Math.}  {\bf189} (1998), 191--207.

\bibitem{kuchen} C.Y. Ku and W. Chen, An analogue of the Gallai-Edmonds structure theorem for non-zero roots of the matching polynomial, {\em J. Combin. Theory Ser. B} {\bf100} (2010), 119--127.

\bibitem{ku}  H. Kunz, Location of the zeros of the partition function for some classical lattice systems, {\em Phys. Letters (A)} {\bf32} (1970), 311--312.
\bibitem{mn} A. de Mier and M. Noy,  On graphs determined by their Tutte polynomials, {\em Graphs Combin.} {\bf20} (2004), 105--119.
\bibitem{n} M. Noy, Graphs determined by polynomial invariants, {\em Theoret. Comput. Sci.} {\bf307} (2003), 365--384.

\end{thebibliography}
\end{document}